\documentclass{amsart}
\usepackage{amsfonts}
\usepackage{graphicx}

\newtheorem{theorem}{Theorem}[section]
\newtheorem*{theorem A}{Theorem A}
\newtheorem*{theorem B}{N\"olker's Theorem}

\theoremstyle{remark}
\newtheorem{remark}{Remark}[section]
\theoremstyle{remark}

\theoremstyle{definition}

\numberwithin{equation}{section}
\def\({\left ( }
\def\){\right )}
\def\<{\left < }
\def\>{\right >}

\input{tcilatex}

\begin{document}
\title[Classifications of Translation Surfaces in Isotropic Geometry] {Classifications of translation surfaces in isotropic geometry with constant curvature}
\author{Muhittin Evren Aydin}
\address{Department of Mathematics, Faculty of Science, Firat University, Elazig, 23200, Turkey}
\email{meaydin@firat.edu.tr}
\thanks{}

\begin{abstract}
We classify translation surfaces in isotropic geometry with \linebreak arbitrary constant isotropic Gaussian and mean curvature under the\linebreak condition that at least one of translating curves lies in a plane.
\end{abstract}

\maketitle

\section{Introduction}
A \textit{translation surface} in a Euclidean space $\mathbb{R}^{3}$ that is generated by translating one curve along another
one locally has the form 
\begin{equation}
r\left( x,y\right) =\alpha \left( x\right)
+ \beta \left( y\right) ,
\end{equation}
where $\alpha $, $\beta $ are \textit{%
translating curves}, see \cite{5}. The recent developments relating to such surfaces in $\mathbb{R}^{3}$  with constant Gaussian and mean curvature were well-structured in Lopez and Moruz's work \cite{16}.

 If $\alpha$ and $\beta$ lie in orthogonal planes then, by a
coordinate change, the surface can be locally described in explicit form
\begin{equation}
z\left( x,y\right) =f\left( x\right) +g\left( y\right) ,
\end{equation}%
where $f,g$ are smooth real-valued functions of one variable. In this case, except the planes, only minimal translation surface (i.e. mean curvature vanishes identically) is the \textit{Scherk surface}, namely the graph of (\cite{25})
\begin{equation*}
z\left( x,y\right) =\frac{1}{c}\log \left\vert \frac{\cos \left(
cy\right) }{\cos \left( cx\right) }\right\vert  ,\text{ }c\in \mathbb{R-}\left\{ 0\right\} .
\end{equation*}%
The extensions of this concept in (semi-) Euclidean and homogeneous spaces can be found in \cite{6}-\cite{8}, \cite{10}-\cite{12}, \cite{15,17,19,20,26,28,29}.

Most recently, the notion of \textit{affine translation surfaces} in $\mathbb{R}^{3}$ were introduced by Liu and Yu \cite{13} as the graphs of 
\begin{equation}
z\left( x,y\right) =f\left( x\right) +g\left( y+ax\right) , \text{ }a\in \mathbb{R-}\left\{ 0\right\} . 
\end{equation}%
By the coordinate change $x=u$, $y=v-au$ in (1.3) one can be parametrized by
\begin{equation*}
r\left( u,v\right) =\left( u,v-au,f\left( u\right) +g\left( v\right) \right),
\end{equation*}%
where the translating curves are in the planes $x=0$ and $%
ax+y=0.$ Since $a\neq 0,$ these planes are not orthogonal and the obtained
surface is a natural \linebreak generalization of (1.2). In same paper, the authors conjectured that, expect planes, only minimal affine translation surface, usually called \textit{affine Scherk surface}, is the graph of the function (also see \cite{14,30})%
\begin{equation*}
z\left( x,y\right) =\frac{1}{c} \log \left\vert \frac{\cos \left( c%
\sqrt{1+a^{2}}x\right) }{\cos \left( c\left[ y+ax\right] \right) }%
\right\vert ,\text{ }a,c\in \mathbb{R-}\left\{ 0\right\}.
\end{equation*}

In the present paper, by following this idea, we introduce and classify a new type of affine translation surfaces in isotropic geometry with constant isotropic Gaussian curvature (CIGC) and consant isotropic mean curvature (CIMC). In addition, we obtain the surfaces with CIGC and CIMC whose one translating curve is planar and other one space curve.

\section{Preliminaries}

For fundamental notions of curves and surfaces in isotropic geometry that is one of the Cayley-Klein geometries, we refer the reader to \cite{3,4,9}, \cite{21}-\cite{24}. It can be described via the projective space as next paragraph:

Let $\mathbb{P}^{3}$ denote the projective space and $\Gamma $ a plane in $\mathbb{%
P}^{3}.$ Then an affine space can be obtained from $\mathbb{P}^{3}$ by
substracting $\Gamma $ which we call \textit{absolute plane.} If $\Gamma $
involves a pair of complex-conjugate straight lines $l_{1}$ and $l_{2},$
so-called the \textit{absolute lines, }then the obtained affine space becomes an 
\textit{isotropic space} $\mathbb{I}^{3},$ where the triple $\left( \Gamma
,l_{1},l_{2}\right) $ is known as the \textit{absolute figure} of $\mathbb{I}^{3}$.

Let a quadruple $\left( \tilde{t}:\tilde{x}:\tilde{y}:\tilde{z}\right) $ be the
projective coordinates, i.e. $\left( \tilde{t}:\tilde{x}:\tilde{y}:\tilde{z}%
\right) \neq (0:0:0:0).$ Then $\Gamma $ and $l_{1},l_{2}$ are respectively parameterized by $\tilde{t}=0$
and $\tilde{t}=\tilde{x}\pm i\tilde{y}=0.$ The intersection
point of $l_{1}$ and $l_{2}$ is said to be \textit{absolute}, i.e. $\left(
0:0:0:1\right) .$

We will be interested in an affine model of $\mathbb{I}^{3}.$ Thus, by means of the affine coordinates $x=\frac{\tilde{x}}{\tilde{t}},$ $y=%
\frac{\tilde{y}}{\tilde{t}},$ $z=\frac{\tilde{z}}{\tilde{t}},$ $\tilde{t}%
\neq 0,$ group of motions of $\mathbb{I}^{3}$ turns to a six-parameter group given by 
\begin{equation}
\left( x,y,z\right) \longmapsto \left( x^{\prime },y^{\prime },z^{\prime
}\right) :\left\{ 
\begin{array}{l}
x^{\prime }=a+x\cos \theta -y\sin \theta , \\ 
y^{\prime }=b+x\sin \theta +y\cos \theta , \\ 
z^{\prime }=c+dx+ey+z,%
\end{array}%
\right.  \tag{2.1}
\end{equation}%
where $a,b,c,d,e,\theta \in \mathbb{R}.$ The metric invariants of $\mathbb{I}%
^{3}$ under $\left( 2.1\right) ,$ such as \textit{isotropic distance} and 
\textit{angle}, are Euclidean invariants in the Cartesian plane.

A line is said to be \textit{isotropic} provided its point at infinity agrees with the absolute point. In the affine model, those correspond to the lines parallel to $z-$axes. Otherwise, it is called \textit{non-isotropic line}.

The plane involving an isotropic line is said to be \textit{isotropic} and then its line at infinity involves the absolute point. Otherwise it is called \textit{non-isotropic plane}. For example; the equation $ax+by+cz=0$ $\left(a,b,c \in \mathbb{R}, \text{ } c \neq 0 \right)$ determines a non-isotropic plane while the equation $ax+by=0$ determine an isotropic plane.

A unit speed curve has the form 
\begin{equation*}
\alpha :I\subseteq \mathbb{R\longrightarrow I}^{3},\text{ }s\longmapsto
\left( f\left( s\right) ,g\left( s\right) ,h\left( s\right) \right) ,\text{ } \left( f^{\prime } \right ) ^{2}+\left( g^{\prime } \right) ^{2}=1,
\end{equation*}%
where the derivative with respect to $s$ is denoted by a prime and thus the {\it curvature} and {\it torsion}
are respectively given by%
\begin{equation*}
\kappa =\sqrt{\left( f^{\prime \prime }\right) ^{2}+\left( g^{\prime \prime
}\right) ^{2}} \text{ or } \kappa=f^{\prime }g^{\prime \prime }-f^{\prime \prime }g^{\prime }%
\end{equation*}
and 
\begin{equation*}
\tau =\frac{\det \left( \alpha ^{\prime },\alpha ^{\prime \prime
},\alpha ^{\prime \prime \prime }\right) }{\kappa ^{2}} \text{,  } \kappa \neq 0.
\end{equation*}
A curve that lies in an isotropic (resp. non-isotropic) plane is said to be 
\textit{isotropic planar} (resp. \textit{non-isotropic planar}). Otherwise we call it \textit{space curve} and then $\tau \neq 0.$

Let $M^{2}$ be an admissible surface immersed in $\mathbb{I}^{3}$. Then tangent plane $T_{p}\left(
M^{2}\right) $ at each point $p \in M^{2}$ is non-isotropic and thus basically has Euclidean metric. For such a surface, the components $E,F,G$ of the first
fundamental form is obtained by the metric on $M^{2}$ induced from $\mathbb{I}^{3}.$ 

The unit isotropic direction $U=(0,0,1)$ is assumed to be normal vector field of $M^{2}$. Indeed it is
orthogonal to all tangent vectors in $T_{p}\left( M^{2}\right) $ for each $p\in M^{2}.$ Hence the
components of second fundamental form are computed with respect to $U$, namely
\begin{equation*}
l=\frac{\det \left( r_{xx},r_{x},r_{y}\right) }{\sqrt{%
EG-F^2}},\text{ }m=\frac{\det \left( r_{xy},r_{x},r_{y}\right) }{\sqrt{%
EG-F^2}},\text{ }n=\frac{\det \left( r_{yy},r_{x},r_{y}\right) }{\sqrt{%
EG-F^2 }}, 
\end{equation*}
where $r(x,y)$ is a local parameterization on $M^{2}$.

The \textit{isotropic Gaussian} (so-called \textit{relative})
and \textit{mean curvature} are respectively defined by%
\begin{equation*}
K=\frac{ln-m^{2}}{EG-F^{2}},\text{ }H=\frac{En-2Fm+Gl}{2\left(
EG-F^{2}\right) }.
\end{equation*}%
A surface for which $H$ (resp. $K$) vanishes identically is said to be \textit{isotropic
minimal} (resp. \textit{isotropic flat}). Morever, a surface is said to have \textit{CIMC} (resp. \textit{CIGC}) if $H$  (resp. $K$) is a constant function on whole surface.

In order to provide convenience in the calculations, we shall denote nonzero constants by $c_{1},c_{2},...$ and some constants by $d_{1},d_{2},...$, $(i=1,2,...)$ throughout the paper, unless otherwise stated.

\section{Categorisation of  Translation Surfaces}
There are two different types of planes in $\mathbb{I}^{3}$ and thus the translation surfaces of the form (1.1) can be described in terms of translating curves as below:

\begin{description}
\item [Type I] $\alpha $ and $\beta $ are planar.

\begin{description}
\item[Type I.1] $\alpha $ and $\beta $ are isotropic planar.

\item[Type I.2] $\alpha $ is isotropic planar and $\beta $ is non-isotropic
planar.

\item[Type I.3] $\alpha $ and $\beta $ are non-isotropic planar.
\end{description}

\item[Type II] $\alpha $ is isotropic planar and $\beta $ is space curve.

\item [Type III] $\alpha $ is non-isotropic planar and $\beta $ is space curve.

\item[Type IV] $\alpha $ and $\beta $ are space curves.
\end{description}

A surface which belongs to one Type is no equivalent to that of another Type up to the isotropic metric. In the particular case the planes are orthogonal, the surfaces of Type I can be locally given by
\begin{description}
\item[Type I.1*] Both translating curves are isotropic planar ($x=0$ and $y=0)$
\begin{equation*}
r\left( x,y\right) =\left( x,y,f\left( x\right) +g\left(
y\right) \right).
\end{equation*} 

\item[Type I.2*] One translating curve is non-isotropic planar $(z=0)$ and other one isotropic planar ($x=0$)
\begin{equation*}
r\left( x,z\right) =\left( x,f\left( x\right) +g\left(
z\right) ,z\right) .
\end{equation*} 

\item[Type I.3*] Both translating curves are non-isotropic planar ($y-z= \pi$ and $y+z= \pi$)
\begin{equation*}
r\left( y,z\right) =\frac{1}{2}\left( f\left( y\right)
+g\left( z\right) ,y-z+\pi ,y+z\right) .
\end{equation*} 
\end{description}
These surfaces with CIMC and CIGC were obtained in \cite{18,27}. 

If the planes are not orthogonal, according to above categorisation, the surfaces of Type I are so-called \textit{affine translation surfaces}. Let $\left (a_{ij} \right )$, $i,j=1,2$, be a non-orthogonal real matrix and $\det\left (a_{ij} \right ) \neq 0$. More generally the surfaces of Type I.1 are of the form
\begin{equation}
r\left( u,v\right) =\left( \frac{a_{22}u}{\det \left( a_{ij}\right) }-\frac{a_{12}v}{\det \left( a_{ij}\right) },-\frac{a_{21}u}{\det \left( a_{ij}\right) }+\frac{a_{11}v}{\det \left( a_{ij}\right) },f\left( u \right) +g\left( v\right) \right), \tag {3.1}
\end{equation}
After a coordinate change, (3.1) turns to
\begin{equation}
r\left( x,y\right) =\left( x,y,f\left( a_{11}x+a_{12}y \right) +g\left( a_{21}x+a_{22}y\right) \right).
\tag {3.2}
\end{equation}
Such surfaces with CIMC and CIGC, which we call \textit {affine translation surface of first kind}, were presented in \cite{1}. In this paper, we shall be interested in the surfaces of Type I.2-Type III.

Furthermore, in a special case that one curve is isotropic planar ($x=0$) and another one without condition, the surfaces with CIMC and CIGC were provided in \cite{2}.

\section{Surfaces of  Type I.2 and Type I.3 with Constant Curvature}

Denote $\left( a_{ij}\right) $ a non-singular real matrix and $\omega=\det \left( a_{ij}\right) \neq 0,$ $i,j=1,2.$ Then we can consider the translation surface, generated by planar curves, of the form
\begin{equation}
r\left( u,v\right) =\left(  \frac{a_{22}u}{\omega }-\frac{a_{12}v}{\omega },f\left( u \right)+g\left(v \right) ,-\frac{a_{21}u}{\omega }+\frac{a_{11}v}{\omega }\right),   \tag{4.1}
\end{equation}
where the translating curves and the planes involving them are given by
\begin{equation*}
\alpha \left( u\right) =\left( \frac{a_{22}u}{\omega },f\left( u\right) ,-\frac{a_{21}u}{\omega }\right), \text{ } \Gamma _{\alpha}:a_{21}x+a_{22}z=0
\end{equation*}
and 
\begin{equation*}
\beta \left( v\right) =\left( -\frac{a_{12}v}{\omega },g\left( v\right) ,\frac{a_{11}v}{\omega }\right), \text{ } \Gamma _{\beta}:a_{11}x+a_{12}z=0.
\end{equation*}
\begin{remark}
\textit {Since the roles of $f$ and $g$ are symmetric we shall only discuss the cases depending on $f$ throughout the section.} 
\end{remark}

We need to state the following items for the surface of the form (4.1):
\begin{itemize}
\item $\Gamma _{\alpha}$ becomes orthogonal to $\Gamma _{\beta}$ when $\left( a_{ij}\right) $ is an orthogonal matrix.

\item If $a_{12}=0$, then $\Gamma _{\alpha} $ (resp. $\Gamma _{\beta} )$ is a non-isotropic plane (resp. isotropic plane). Thereby the obtained
surface belongs to Type I.2.

\item Otherwise, i.e. $a_{12} \neq 0$, $\Gamma _{\alpha} $ and $\Gamma _{\beta} $ are non-isotropic planes and the \linebreak obtained surface belongs to Type I.3. 
\end{itemize}

If we apply a coordinate change in (4.1) as follows:
\begin{equation}
u =a_{11}x+a_{12}z,\text{ \ }v
=a_{21}x+a_{22}z,  \tag{4.2}
\end{equation}
where $\left( a_{ij}\right) $ is no orthogonal matrix, then it turns to
\begin{equation}
r\left( x,z\right) =\left( x,f\left( u \right) +g\left(
v \right) ,z\right). \tag{4.3}
\end{equation}

The transformations of the form (4.2) are usually called \textit{affine parameter \linebreak coordinates}. Since (4.3) is no equivalent to (3.2) up to the isotropic metric, it may be called \textit{affine translation surface of second kind}. The positive side of this notion is to express the surfaces of Type I.2 and Type I.3 into one format. Also, it seems surprising that only two different kinds of affine translation surfaces appear despite the fact that there are three different kinds of translation surfaces generated by planar curves.

Now we purpose to obtain the surfaces of the form (4.3) with CIGC. For this, the regularity of $\left( 4.3\right) $ implies that 
\begin{equation*}
a_{12}f^{\prime}+a_{22}g^{\prime }\neq 0,\text{ } f^{\prime }=\frac{df}{du}, \text{ } g^{\prime }=\frac{dg}{dv}.
\end{equation*}
By a calculation, $K$ turns to%
\begin{equation}
K=\frac{\omega ^{2}f^{\prime \prime }g^{\prime \prime }}{\left( a_{12}f^{\prime }+a_{22}g^{\prime }\right)
^{4}}.  \tag{4.4}
\end{equation}
It is seen from (4.4) that $K$ vanishes identically provided $f^{\prime \prime }=0$, namely the surface is a \textit{generalized cylinder} (or \textit {cylindrical surface}) with non-isotropic rulings. So, next result can be stated in order for $K$ to be a non-vanishing constant:
\begin{theorem}
For a translation surface in $\mathbb{I}^{3}$ of the form (4.3) with nonzero CIGC $(K_{0})$, up to suitable translations and constants, the following holds:
\begin{equation*}
f\left( u\right) =c_{1}x^{2},\text{ }g\left( v\right) =c_{2}v^{\frac{2}{3}},
\end{equation*}%
where $(u,v)$ is the affine parameter coordinates given by $\left( 4.2\right) .$
\end{theorem}

\begin{proof}
Since $K_{0} \neq 0$ in (4.4), we have $f^{\prime \prime }g^{\prime \prime }\neq 0.$ The partial derivative of $\left( 4.4\right) $ with respect to $u$ gives%
\begin{equation}
\frac{4K_{0}}{\omega ^{2}}\left( a_{12}f^{\prime }+a_{22}g^{\prime }\right)
^{3}\left( a_{12}f^{\prime \prime }\right) =f^{\prime \prime \prime
}g^{\prime \prime }.  \tag{4.5}
\end{equation}%
We distinguish two cases to solve $\left( 4.5\right)$:
\begin{itemize}
\item $a_{12}=0$. Then $a_{11}a_{22}\neq 0$ due to $\omega \neq 0$. By (4.5) we have $f^{\prime \prime }=c_{1}$ and thus $\left( 4.4\right) $ follows
\begin{equation}
\frac{K_{0}a_{22}^{2}}{c_{1}a_{11}^{2}}=\frac{g^{\prime \prime }}{\left(
g^{\prime }\right) ^{4}}.  \tag{4.6}
\end{equation}%
By solving $\left(
4.6\right) $ we obtain%
\begin{equation*}
g\left( v\right) =\frac{-c_{1}a_{11}^{2}}{2K_{0}a_{22}^{2}}\left( -\frac{%
3K_{0}a_{22}^{2}}{c_{1}a_{11}^{2}}v+d_{2}\right) ^{\frac{2}{3}}+d_{3}, 
\end{equation*}%
which gives the hypothesis of the theorem.

\item $a_{12} \neq 0.$ By symmetry,
we deduce $a_{22} \neq 0.$ Then $\left(
4.5\right) $ can be arranged as%
\begin{equation}
\frac{\left( a_{12}f^{\prime }+a_{22}g^{\prime }\right) ^{3}}{g^{\prime
\prime }}=\frac{\omega ^{2}}{4K_{0}a_{12}}\left( \frac{f^{\prime \prime
\prime }}{f^{\prime \prime }}\right) .  \tag{4.7}
\end{equation}%
The partial derivative of $\left( 4.7\right) $ with respect to $v$ leads to%
\begin{equation}
3a_{22}\left( g^{\prime \prime }\right) ^{2}-\left( a_{12}f^{\prime
}+a_{22}g^{\prime }\right) g^{\prime \prime \prime }=0,  \tag{4.8}
\end{equation}%
where $g^{\prime \prime \prime } \neq 0$ due to $a_{22} g^{\prime \prime } \neq 0$. After taking partial derivative of $\left( 4.8\right) $ with respect to $u$
we immediately achieve a contradiction.
\end{itemize} 
\end{proof}

By a calculation, the isotropic mean curvature is%
\begin{equation}
H=-\frac{\left[ a_{12}^{2}+\left( \omega g^{\prime }\right) ^{2}\right]
f^{\prime \prime }+\left[ a_{22}^{2}+\left( \omega f^{\prime }\right) ^{2}%
\right] g^{\prime \prime }}{2\left( a_{12}f^{\prime }+a_{22}g^{\prime
}\right) ^{3}}. \tag{4.9}
\end{equation}
First, we discuss the minimality case via the following result:
\begin{theorem}
For an isotropic minimal translation surface in $\mathbb{I}^{3}$ of the form (4.3), up to suitable translations and constants, one of the following occurs:

\begin{enumerate}
\item[(a)] it is a non-isotropic plane;

\item[(b)] 
$f\left( u \right) =\frac{1}{c_{1}}\log \left\vert \cos \left (c_{2}%
x \right )\right\vert$, $g\left( v\right) =\frac{-1}{c_{1}}\log \left\vert c_{1}v\right\vert$;

\item[(c)] 
$f\left( u\right) =\frac{1}{c_{1}}\log \left\vert \cos \left( c_{2}u\right)
\right\vert$, $g\left( v\right) =\frac{-1}{c_{1}}\log \left\vert \cos
\left( c_{3}v\right) \right\vert$,
 \end{enumerate}
where $(u,v)$ is the affine parameter coordinates given by $\left( 4.2\right) .$
\end{theorem}

\begin{proof}
Since $H$ vanishes identically, $\left( 4.9\right) $ reduces to%
\begin{equation}
\left[ a_{12}^{2}+\left( \omega g^{\prime }\right) ^{2}\right] f^{\prime
\prime }+\left[ a_{22}^{2}+\left( \omega f^{\prime }\right) ^{2}\right]
g^{\prime \prime }=0.  \tag{4.10}
\end{equation}%
$f^{\prime \prime }=g^{\prime \prime }=0$ is a solution for $\left( 4.10\right) ,
$ and in this case, the surface is a non-isotropic plane. Suppose that $f^{\prime
\prime }g^{\prime \prime }\neq 0.$ Hence $\left( 4.10\right) $ implies%
\begin{equation}
-\frac{f^{\prime \prime }}{a_{22}^{2}+\left( \omega f^{\prime }\right) ^{2}}%
=c_{1}=\frac{g^{\prime \prime }}{a_{12}^{2}+\left( \omega g^{\prime
}\right) ^{2}}.  \tag{4.11}
\end{equation}%
We have to discuss two cases to solve (4.11).
\begin{itemize}
\item $a_{12}=0.$ Then we have $a_{11}a_{22}\neq 0$ because of $\omega \neq 0$. By solving $\left( 4.11\right) $ we
conclude%
\begin{equation*}
f\left( a_{11}x\right) =\frac{1}{c_{1}a_{11}^{2}a_{22}^{2}}\log \left\vert \cos
\left( c_{1}a_{11}^{2}a_{22}^{2}x+d_{2}\right) \right\vert +d_{3},
\end{equation*}
and%
\begin{equation*}
g\left( v\right) =\frac{-1}{c_{1}a_{11}^{2}a_{22}^{2}}\log \left\vert
c_{1}a_{11}^{2}a_{22}^{2}v+d_{4}\right\vert +d_{5}.
\end{equation*}%
This gives the item (b) of the theorem.

\item $a_{12}\neq 0.$ By symmetry, we get $a_{22}\neq 0$ and
solving $\left( 4.11\right) $ leads to%
\begin{equation*}
f\left( u\right) =\frac{1}{c_{1}\omega ^{2}}\log \left\vert \cos \left(
\omega c_{1}a_{22}u+c_{2}\right) \right\vert +c_{3}
\end{equation*}%
and%
\begin{equation*}
g\left( v\right) =\frac{-1}{c_{1}\omega ^{2}}\log \left\vert \cos \left(
\omega c_{1}a_{12}v+d_{4}\right) \right\vert +d_{5}.
\end{equation*}%
Therefore the proof of the theorem is completed.
\end{itemize}
\end{proof}

\begin{theorem}
For a translation surface in $\mathbb{I}^{3}$ of the form (4.3) with nonzero CIMC $(H_{0})$, up to suitable translations and constants, we have either

\begin{enumerate}
\item[(a)] 
$f\left( u\right) =c_{1}u,\text{ \ }g\left( v \right) =c_{2}x^{2}$, or

\item[(b)] 
$f\left( u\right) =c_{1}u,$ $g\left( v\right) =c_{2}v^{\frac{1}{2}%
}+c_{3}v$, 
\end{enumerate}
where $\left( u,v\right) $ is the affine parameter coordinates given by $%
\left( 4.1\right) .$
\end{theorem}

\begin{proof}
The partial derivative of $\left( 4.9\right) $ with
respect to $u$ and $v$ yields%
\begin{equation}
-6H_{0}\omega ^{-2}a_{12}a_{22}\left( a_{12}f^{\prime }+a_{22}g^{\prime
}\right) \left( f^{\prime \prime }g^{\prime \prime }\right) =g^{\prime
}g^{\prime \prime }f^{\prime \prime \prime }+f^{\prime }f^{\prime \prime
}g^{\prime \prime \prime }.  \tag{4.12}
\end{equation}%
The situation for which both $f^{\prime \prime }$ and $g^{\prime \prime }$ vanish is a solution for $\left( 4.12\right) ,$ however we omit this one since $%
H_{0}\neq 0.$ We distinguish the remaining cases:
\begin{itemize}
\item $f=c_{1}u+d_{1}$ and $g^{\prime \prime }\neq 0$.  This assumption is a solution for (4.12). Then from $\left(
4.9\right) ,$ we derive%
\begin{equation}
\frac{g^{\prime \prime }}{\left( a_{12}c_{1}+a_{22}g^{\prime }\right) ^{3}}=%
\frac{-2H_{0}}{a_{22}^{2}+\left( \omega c_{1}\right) ^{2}}.  \tag{4.13}
\end{equation}%
We have two cases:
\begin{itemize}
\item[\textbf{(1)}] $a_{22}=0.$ Solving $\left( 4.13\right) $ leads to%
\begin{equation*}
g\left( a_{21}x\right) =-H_{0}a_{12}c_{1}x^{2}+d_{2}x+d_{3},
\end{equation*}%
where $a_{12}a_{21}\neq 0$ due to $\omega \neq 0.$ This implies the item
(a) of the theorem.

\item[\textbf{(2)}] $a_{22}\neq 0.$ By symmetry we have $a_{12}\neq 0.$ Solving $%
\left( 4.13\right) $ gives%
\begin{equation*}
g\left( v\right) =\frac{a_{22}^{2}+\left( \omega c_{1}\right) ^{2}}{%
2H_{0}a_{22}^{2}}\left( \frac{4H_{0}a_{22}}{a_{22}^{2}+\left( \omega
c_{1}\right) ^{2}}v+d_{4}\right) ^{\frac{1}{2}}-\frac{a_{12}c_{1}}{a_{22}}%
v+d_{5},
\end{equation*}%
which proves the item (b) of the theorem.
\end{itemize}

\item $f^{\prime \prime }g^{\prime \prime }\neq 0.$ By dividing 
$\left( 4.12\right) $ with $f^{\prime \prime }g^{\prime \prime },$ one can
be rewritten as%
\begin{equation}
-6H_{0}\omega ^{-2}a_{12}a_{22}\left( a_{12}f^{\prime }+a_{22}g^{\prime
}\right) =g^{\prime }\frac{f^{\prime \prime \prime }}{f^{\prime \prime }}%
+f^{\prime }\frac{g^{\prime \prime \prime }}{g^{\prime \prime }}.  \tag{4.14}
\end{equation}%
We have again cases:
\begin{itemize}
\item[(\textbf{1)}] $a_{12}=0.$ $\omega \neq 0$ implies $a_{11}a_{22}\neq 0.$ $%
\left( 4.14\right) $ turns to%
\begin{equation*}
\frac{f^{\prime \prime \prime }}{f^{\prime }f^{\prime \prime }}=d_{1}=-\frac{%
g^{\prime \prime \prime }}{g^{\prime }g^{\prime \prime }},
\end{equation*}%
which gives that $f^{\prime \prime }=c_{1}e^{d_{1}f}$ and $g^{\prime
\prime }=c_{2}e^{-d_{1}g}$. By substituting those in $\left( 4.9\right) $ we derive
\begin{equation}
-2H_{0}a_{22}\left( g^{\prime }\right) ^{3}=c_{1}a_{11}^{2}\left( g^{\prime }\right)
^{2}e^{d_{1}f}+\left[ c_{2}+c_{2}a_{11}^{2}\left( f^{\prime }\right) ^{2}\right]
e^{-d_{1}g}.  \tag{4.15}
\end{equation}%
Put $f^{\prime }=p$ and $g^{\prime }=q$ in (4.15). Then taking partial derivative of $%
\left( 4.15\right) $ with respect to $f$ yields%
\begin{equation}
0=d_{1}c_{1}q^{2}e^{d_{1}f}+2c_{2}pp^{\prime }e^{-d_{1}g},  \tag{4.16}
\end{equation}%
where $p^{\prime }=\frac{dp}{df}=\frac{f^{\prime \prime }}{f^{\prime }}.$ If 
$d_{1}=0$ in $\left( 4.16\right) $ then we obtain the contradiction $%
p^{\prime }=0.$ Otherwise we have 
\begin{equation}
\frac{d_{1}c_{1}e^{d_{1}f}}{2c_{2}pp^{\prime }}=c_{3}=-\frac{e^{-d_{1}g}}{%
q^{2}}.  \tag{4.17}
\end{equation}%
By substituting the second equality in $\left( 4.17\right) $ into $\left(
4.15\right) ,$ we conclude%
\begin{equation}
-2H_{0}a_{22}q\left( g\right) =c_{1}a_{11}^{2}e^{d_{1}f}-c_{2}c_{3}\left[ 1+a_{11}^{2}p\left(
f\right) ^{2}\right] .  \tag{4.18}
\end{equation}%
The left side in $\left( 4.18\right) $ is a function of $g$ however other
side is a function of $f.$ This is not possible.

\item[\textbf{(2)}] $a_{12}\neq 0$ in $\left( 4.14\right) .$ The symmetry
implies $a_{22}\neq 0.$ By dividing $\left( 4.14\right) $ with $f^{\prime
}g^{\prime },$ we write%
\begin{equation}
D\left( \frac{a_{12}}{g^{\prime }}+\frac{a_{22}}{f^{\prime }}\right) =\frac{%
f^{\prime \prime \prime }}{f^{\prime }f^{\prime \prime }}+\frac{g^{\prime
\prime \prime }}{g^{\prime }g^{\prime \prime }},  \tag{4.19}
\end{equation}%
where $D=-6H_{0}\omega ^{-2}a_{12}a_{22}.$ $\left( 4.19\right) $ follows%
\begin{equation}
f^{\prime \prime \prime }=\left( -d_{1}f^{\prime }+Da_{22}\right) f^{\prime
\prime }\text{ and }g^{\prime \prime \prime }=\left( d_{1}g^{\prime
}+Da_{12}\right) g^{\prime \prime }.  \tag{4.20}
\end{equation}%
On the other hand by taking partial derivative of $\left( 4.9\right) $ with
respect to $v$ and considering the second equality in $\left( 4.20\right) $ we obtain%
\begin{equation}
-2Ha_{22}\left( a_{12}f^{\prime }+a_{22}g^{\prime }\right) ^{2}=2\omega
^{2}g^{\prime }f^{\prime \prime }+\left[ a_{22}^{2}+\left( \omega f^{\prime
}\right) ^{2}\right] \left( d_{1}g^{\prime }+Da_{12}\right) .  \tag{4.21}
\end{equation}%
Taking twice partial derivative of $\left( 4.21\right) $ with respect to $v$
gives the contradiction $g^{\prime \prime }=0.$
\end{itemize}

\end{itemize}

\end{proof}
\section{Surfaces of  Type II with Constant Curvature}

Let $\alpha $ be an isotropic planar curve  and $\beta $ a space curve
given by %
\begin{equation*}
\alpha \left( x\right) =\left( x,ax,f\left( x\right) \right) ,\text{ \ }%
\beta \left( y\right) =\left( y,g\left( y\right) ,h\left( y\right) \right) ,
\end{equation*}%
where $a\in \mathbb{R}$. Since the torsion of $\beta$ is nonzero function, we deduce that%
\begin{equation}
g^{\prime \prime }h^{\prime \prime \prime }-g^{\prime \prime
\prime }h^{\prime \prime }\neq 0,  \tag{5.1}
\end{equation}%
where $g^{\prime }=\frac{dg}{dy},$ $h^{\prime }=\frac{dh}{dy}$ and so on.
Then the obtained translation surface belongs to Type II and has the form%
\begin{equation}
r\left( x,y\right) =\left( x+y,ax+g\left( y\right) ,f\left( x\right)
+h\left( y\right) \right) .  \tag{5.2}
\end{equation}%
The assumption (5.1) ensures the regularity of (5.2), i.e. $g^{\prime
}-a\neq 0.$ Hence, by a calculation, the Gaussian curvature turns to%
\begin{equation}
K=\frac{f^{\prime \prime }\left[ h^{\prime \prime }\left( g^{\prime
}-a\right) -g^{\prime \prime }\left( h^{\prime }-f^{\prime }\right) \right] 
}{\left( g^{\prime }-a\right) ^{3}}, \tag{5.3}
\end{equation}
where $f^{\prime }=\frac{df}{dx}$, etc. 
\begin{theorem}
A translation surface in $\mathbb{I}^{3}$ of the form (5.2) with CIGC ($K_{0}$) is a generalized
cylinder with non-isotropic rulings, i.e. $K_{0}=0$.
\end{theorem}

\begin{proof}
Assume that $K_{0}=0$. In this case, only possibility is that $f^{\prime\prime }=0$, namely, $\alpha $ is a non-isotropic line. Otherwise, $f^{\prime\prime } \neq 0$, we have $h^{\prime \prime }\left( g^{\prime}-a\right)= g^{\prime \prime }\left( h^{\prime }-f^{\prime }\right)$. Taking partial derivative of this one with respect to $x$ yields the contradiction $g^{\prime \prime}=0$ due to (5.1). If $K_{0} \neq 0$, then (5.3) can be rewritten as%
\begin{equation}
\frac{K_{0}}{f^{\prime \prime }}=\frac{h^{\prime \prime }}{\left( g^{\prime
}-a\right) ^{2}}-\frac{g^{\prime \prime }}{\left( g^{\prime }-a\right) ^{3}}%
\left( h^{\prime }-f^{\prime }\right) .  \tag{5.4}
\end{equation}%
Taking partial derivative of (5.4) with respect to $x$ yields%
\begin{equation*}
-K_{0}\frac{f^{\prime \prime \prime }}{\left( f^{\prime \prime }\right) ^{3}}%
=\frac{g^{\prime \prime }}{\left( g^{\prime }-a\right) ^{3}}
\end{equation*}%
and solving this one%
\begin{equation}
f\left( x\right) =\frac{1}{3c_{1}^{2}}\left( -2c_{1}x+d_{1}\right) ^{\frac{3}{2}%
}+d_{2}x+d_{3}  \tag{5.5}
\end{equation}%
and 
\begin{equation}
g\left( y\right) =\frac{1}{K_{0}c_{1}}\left( 2K_{0}c_{1}y+d_{4}\right) ^{%
\frac{1}{2}}+ay+d_{5}.  \tag{5.6}
\end{equation}%
Substituting (5.5) and (5.6) into (5.4) gives%
\begin{equation}
0=\frac{h^{\prime \prime }}{h^{\prime }+d_{2}}+\frac{K_{0}c_{1}}{%
2K_{0}c_{1}y+d_{4}}.  \tag{5.7}
\end{equation}%
By solving (5.7), we find%
\begin{equation}
h\left( y\right) =\frac{c_{2}}{K_{0}c_{1}}\left( 2K_{0}c_{1}y+d_{4}\right) ^{%
\frac{1}{2}}-d_{2}y+d_{6}.  \tag{5.8}
\end{equation}%
Combining (5.6) with (5.8) gives a contradiction due to (5.1).
\end{proof}

By a direct calculation, the mean curvature is%
\begin{equation}
2H=\frac{\left[ 1+\left( g^{\prime }\right) ^{2}\right] \left( g^{\prime
}-a\right) f^{\prime \prime }+\left( 1+a^{2}\right) \left[ h^{\prime \prime
}\left( g^{\prime }-a\right) -g^{\prime \prime }\left( h^{\prime }-f^{\prime
}\right) \right] }{\left( g^{\prime }-a\right) ^{3}}.  \tag{5.9}
\end{equation}
First we distinguish the minimality case: 

\begin{theorem}
A translation surface in $\mathbb{I}^{3}$ of the form (5.2) cannot be
isotropic minimal.
\end{theorem}

\begin{proof}
We prove it by contradiction. If $H=0$, then
(5.9) reduces to%
\begin{equation}
\left[ 1+\left( g^{\prime }\right) ^{2}\right] \left( g^{\prime }-a\right)
f^{\prime \prime }+\left( 1+a^{2}\right) \left[ h^{\prime \prime }\left(
g^{\prime }-a\right) -g^{\prime \prime }\left( h^{\prime }-f^{\prime
}\right) \right] =0.  \tag{5.10}
\end{equation}%
The partial derivative of (5.10) with respect to $x$ yields%
\begin{equation}
\left[ 1+\left( g^{\prime }\right) ^{2}\right] \left( g^{\prime }-a\right)
f^{\prime \prime \prime }+\left( 1+a^{2}\right) g^{\prime \prime }f^{\prime
\prime }=0.  \tag{5.11}
\end{equation}%
We have to discuss two cases:

\begin{itemize}
\item $f^{\prime \prime }=0,$ i.e. $f\left( x\right) =d_{1}x+d_{2}.$ By
(5.10) we deduce%
\begin{equation*}
\frac{h^{\prime \prime }}{h^{\prime }-d_{1}}=\frac{g^{\prime \prime }}{%
g^{\prime }-a},
\end{equation*}%
which implies%
\begin{equation*}
h=c_{1}g+\left( d_{1}-ac_{1}\right) y-d_{3}.
\end{equation*}%
This is not possible due to (5.1).

\item $f^{\prime \prime }\neq 0.$ (5.11) implies%
\begin{equation}
\frac{f^{\prime \prime \prime }}{f^{\prime \prime }}=-\frac{\left(
1+a^{2}\right) g^{\prime \prime }}{\left[ 1+\left( g^{\prime }\right) ^{2}%
\right] \left( g^{\prime }-a\right) },  \tag{5.12}
\end{equation}%
Then (5.12) follows%
\begin{equation}
f^{\prime \prime }=c_{1}f^{\prime }+d_{1},\text{ \ }\left[ 1+\left(
g^{\prime }\right) ^{2}\right] \left( g^{\prime }-a\right) c_{1}=-\left(
1+a^{2}\right) g^{\prime \prime }.  \tag{5.13}
\end{equation}%
By considering (5.13) into (5.10), we conclude%
\begin{equation}
0=\frac{g^{\prime \prime }}{g^{\prime }-a}-\frac{h^{\prime \prime }}{%
h^{\prime }+\frac{d_{1}}{c_{1}}}.  \tag{5.14}
\end{equation}%
Solving (5.14) gives%
\begin{equation*}
g=c_{2}h+\left( a+\frac{c_{2}d_{1}}{c_{1}}\right) y+d_{2},
\end{equation*}%
which is a contradiction due to (5.1).
\end{itemize}
\end{proof}

\begin{theorem}
For a translation surface in $\mathbb{I}^{3}$ of the form (5.2) has nonzero
CIMC ($H_{0}$), up to suitable
translations and constants, one of the following occurs:
\begin{enumerate}
\item[(a)] it is a generalized cylinder with non-isotropic rulings over the curve 
\begin{equation*}
\beta \left( y\right) =\left( y,g\left(
y\right) ,\frac{H_{0}}{1+a^{2}}\left[ g-ay\right] ^{2}+c_{1}y\right) 
\end{equation*}
\item[(b)] $\alpha \left( x\right) =\left( x,ax,c_{1}\exp \left(
c_{2}x\right) \right) $ and $\beta \left( y\right) =\left( y,g\left(
y\right) ,\frac{H_{0}}{1+a^{2}}\left[ g-ay\right] ^{2}\right) ,$
\end{enumerate}
where $g$ is a non-linear function.
\end{theorem}

\begin{proof}
We seperate the proof into two cases:

\begin{itemize}
\item $f^{\prime \prime }=0,$ $f\left( x\right) =d_{1}x+d_{2}.$ By
substituting it into $\left( 5.9\right) $ we derive%
\begin{equation}
\frac{2H_{0}}{\left( 1+a^{2}\right) }\left( g^{\prime }-a\right) =\left( 
\frac{h^{\prime }-d_{1}}{g^{\prime }-a}\right) ^{\prime }.  \tag{5.15}
\end{equation}%
Twice integration in $\left( 5.15\right) $ yields%
\begin{equation*}
h=\frac{H_{0}}{\left( 1+a^{2}\right) }\left( g-ay\right) ^{2}+d_{1}y+d_{3},
\end{equation*}%
which implies item (a) of the theorem.

\item $f^{\prime \prime }\neq 0.$ By taking partial derivative of (5.9)
with respect to $x$ we obtain (5.11). It means that next steps will be
similar to that of Theorem 5.2. Thus we have (5.13), namely%
\begin{equation*}
f\left( x\right) =c_{2}\exp \left( c_{1}x\right) +d_{1}x+d_{2}
\end{equation*}%
and 
\begin{equation*}
\text{\ }\left[ 1+\left( g^{\prime }\right) ^{2}\right] \left( g^{\prime
}-a\right) c_{1}=-\left( 1+a^{2}\right) g^{\prime \prime }.
\end{equation*}%
Substituting those into (5.9), we conclude%
\begin{equation}
\frac{2H_{0}}{\left( 1+a^{2}\right) }\left( g^{\prime }-a\right) =\frac{%
d_{1}g^{\prime \prime }}{c_{1}\left( g^{\prime }-a\right) ^{2}}+\left( \frac{%
h^{\prime }}{g^{\prime }-a}\right) ^{\prime }.  \tag{5.16}
\end{equation}%
Twice integration in $\left(5.16\right) $ gives%
\begin{equation*}
h=\frac{H_{0}}{\left( 1+a^{2}\right) }\left( g-ay\right) ^{2}-\frac{d_{1}}{%
c_{1}}y+d_{3}(g-ay)+d_{4},
\end{equation*}%
which completes the proof.
\end{itemize}
\end{proof}

\section{Surfaces of  Type III with Constant Curvature}

Let $\alpha $ be a non-isotropic planar curve  and $\beta $ a space curve
given by %
\begin{equation*}
\alpha \left( x\right) =\left( x,f\left( x\right) ,ax\right) ,\text{ \ }%
\beta \left( y\right) =\left( y,g\left( y\right) ,h\left( y\right) \right) ,
\end{equation*}%
where $a\in \mathbb{R}$. Then the torsion of $\beta$ is nonzero function, namely%
\begin{equation}
g^{\prime \prime }h^{\prime \prime \prime }-g^{\prime \prime
\prime }h^{\prime \prime }\neq 0,  \tag{6.1}
\end{equation}%
where $g^{\prime }=\frac{dg}{dy},$ $h^{\prime }=\frac{dh}{dy}$ and so on.
Hence the surface obtained by a translation of $\alpha$ and $\beta$ belongs to Type III and of the form%
\begin{equation}
r\left( x,y\right) =\left( x+y,f\left( x\right) +g\left( y\right)
,ax+h\left( y\right) \right) .  \tag{6.2}
\end{equation}%
It implies from (6.1) that the surface is regular, i.e. $g^{\prime
}-f^{\prime }\neq 0, f^{\prime }=\frac{df}{dx}.$ By a calculation, the isotropic Gaussian curvature is%
\begin{equation}
K=-\frac{f^{\prime \prime }\left( h^{\prime }-a\right) \left[ h^{\prime
\prime }\left( g^{\prime }-f^{\prime }\right) -g^{\prime \prime }\left(
h^{\prime }-a\right) \right] }{\left( g^{\prime }-f^{\prime }\right) ^{4}}. 
\tag{6.3}
\end{equation}

\begin{theorem}
A translation surface in $\mathbb{I}^{3}$ of the form (6.2) with CIGC ($K_{0}$) is a generalized
cylinder with non-isotropic rulings, namely $K_{0}=0$.
\end{theorem}

\begin{proof}
Since $\beta $ is a space curve, only possibility in (6.3) that $K_{0}$ vanishes is to be $f^{\prime
\prime }=0,$ namely $\alpha $ is a non-isotropic line. Assume that $K$ is a nonzero constant $K_{0}$, i.e. $f^{\prime \prime} \neq 0$ . Then the partial
derivative of (6.3) with respect to $x$ yields%
\begin{equation}
K_{0}\left[ \frac{-4\left( g^{\prime }-f^{\prime }\right) ^{3}}{f^{\prime
\prime }}-\frac{\left( g^{\prime }-f^{\prime }\right) ^{4}f^{\prime \prime
\prime }}{\left( f^{\prime \prime }\right) ^{3}}\right] =-\left( h^{\prime
}-a\right) h^{\prime \prime }.  \tag{6.4}
\end{equation}%
We have two cases:

\begin{itemize}
\item $f^{\prime \prime \prime }=0,$ $f^{\prime \prime }=c_{1}.$ Then from
(6.4), we get%
\begin{equation}
4K_{0}\left( g^{\prime }-f^{\prime }\right) ^{3}=c_{1}\left( h^{\prime
}-a\right) h^{\prime \prime }.  \tag{6.5}
\end{equation}%
The partial derivative of (6.5) with respect to $x$ gives $f^{\prime \prime
}=0,$ which is not our case due to $K_{0}\neq 0.$

\item $f^{\prime \prime \prime }\neq 0.$ Taking partial derivative of (6.4) with respect to $x$ and after dividing
with $\left( g^{\prime }-f^{\prime }\right) ^{4}$ gives%
\begin{equation}
\left[ \frac{-12}{\left( g^{\prime }-f^{\prime }\right) ^{2}}-\frac{%
4f^{\prime \prime \prime }}{\left( g^{\prime }-f^{\prime }\right) \left(
f^{\prime \prime }\right) ^{2}}\right] +\left[ \frac{4f^{\prime \prime
\prime }}{(f^{\prime \prime })^{2}\left( g^{\prime }-f^{\prime }\right) ^{3}}%
-\left( \frac{f^{\prime \prime \prime }}{\left( f^{\prime \prime }\right)
^{3}}\right) ^{\prime }\right] =0.  \tag{6.6}
\end{equation}%
Taking partial derivative of (6.6) with respect to $y$ and after producting
with $\frac{\left( g^{\prime }-f^{\prime }\right) ^{4}}{g^{\prime \prime }}$
gives the following polynomial equation on $g^{\prime}$:%
\begin{equation}
\frac{f^{\prime \prime \prime }}{\left( f^{\prime \prime }\right) ^{2}}%
\left( g^{\prime }\right) ^{2}+\left[ 6-\frac{2f^{\prime }f^{\prime \prime
\prime }}{\left( f^{\prime \prime }\right) ^{2}}\right] g^{\prime
}+f^{\prime \prime \prime }-6f^{\prime }-3\frac{f^{\prime \prime \prime }}{%
f^{\prime \prime }}=0.  \tag{6.7}
\end{equation}%
This is a contradiction since the coefficient of the term $\left ( g^{\prime } \right)^{2}$ in (6.7) cannot vanish.
\end{itemize}
\end{proof}

By a direct calculation, the isotropic mean curvature is%
\begin{equation}
2H=\frac{\left[ 1+\left( f^{\prime }\right) ^{2}\right] \left[ h^{\prime
\prime }\left( g^{\prime }-f^{\prime }\right) -g^{\prime \prime }\left(
h^{\prime }-a\right) \right] -\left[ 1+\left( g^{\prime }\right) ^{2}\right]
\left( h^{\prime }-a\right) f^{\prime \prime }}{\left( g^{\prime }-f^{\prime
}\right) ^{3}}.  \tag{6.8}
\end{equation}%
First we distinguish the minimality case:

\begin{theorem}
A translation surface in $\mathbb{I}^{3}$ of the form (6.2) cannot be
isotropic minimal.
\end{theorem}

\begin{proof}
We prove it by contradiction. If the surface is isotropic minimal, then
(6.8) reduces to%
\begin{equation}
\left[ 1+\left( f^{\prime }\right) ^{2}\right] \left[ h^{\prime \prime
}\left( g^{\prime }-f^{\prime }\right) -g^{\prime \prime }\left( h^{\prime
}-a\right) \right] -\left[ 1+\left( g^{\prime }\right) ^{2}\right] \left(
h^{\prime }-a\right) f^{\prime \prime }=0.  \tag{6.9}
\end{equation}
We have two cases:

\begin{itemize}
\item $f^{\prime \prime }=0,$ $f=d_{1}x+d_{2}.$ Then (6.9) reduces to%
\begin{equation*}
h^{\prime \prime }\left( g^{\prime }-d_{1}\right) -g^{\prime \prime }\left(
h^{\prime }-a\right) =0.
\end{equation*}%
However, this is not possible due to (6.1).

\item $f^{\prime \prime }\neq 0.$ By dividing $\left(6.9\right) $ with $%
\left[ 1+\left( g^{\prime }\right) ^{2}\right] \left[ 1+\left( f^{\prime
}\right) ^{2}\right] \left( h^{\prime }-a\right) ,$ we derive 
\begin{equation}
\frac{h^{\prime \prime }\left( g^{\prime }-f^{\prime }\right) }{\left(
h^{\prime }-a\right) \left[ 1+\left( g^{\prime }\right) ^{2}\right] }+\frac{%
f^{\prime \prime }}{1+\left( f^{\prime }\right) ^{2}}-\frac{g^{\prime \prime
}}{1+\left( g^{\prime }\right) ^{2}}=0.  \tag{6.10}
\end{equation}%
The partial derivative of (6.10) with respect to $x$ and $y$ yields%
\begin{equation}
\frac{h^{\prime \prime }}{\left( h^{\prime }-a\right) \left[ 1+\left(
g^{\prime }\right) ^{2}\right] }=c_{1}.  \tag{6.11}
\end{equation}%
Substituting (6.11) into (6.10) gives%
\begin{equation}
\frac{f^{\prime \prime }}{1+\left( f^{\prime }\right) ^{2}}-c_{1}f^{\prime
}=d_{1}\text{ and }\frac{g^{\prime \prime }}{1+\left( g^{\prime }\right) ^{2}%
}-c_{1}g^{\prime }=d_{1}.  \tag{6.12}
\end{equation}%
By considering the second equality in (6.12) into (6.11), we obtain%
\begin{equation}
\frac{h^{\prime \prime }}{h^{\prime }-a }=\frac{c_{1}g^{\prime
\prime }}{c_{1}g^{\prime }+d_{1}}.  \tag{6.13}
\end{equation}%
Solving (6.13) implies a contradiction due to (6.1).
\end{itemize}
\end{proof}

\begin{theorem}
A translation surface in $\mathbb{I}^{3}$ of the form (6.2) with nonzero
CIMC $H_{0}$ is, up to suitable
translations and constants, a generalized cylinder with non-isotropic rulings over the curve
\begin{equation*}
\beta \left( y\right)
=\left( y,g\left( y\right) ,H_{0} g^{2}(y)+ay\right) , 
\end{equation*}
where $g$ is a non-linear function.
\end{theorem}

\begin{proof}
We divide the proof into two cases:
\begin{itemize}
\item $f^{\prime \prime }=0.$ Then $f\left( x\right)
=d_{1}x+d_{2}$ and (6.8) reduces to
\begin{equation}
\frac{2H}{1+d_{1}^{2}}\left( g^{\prime }-d_{1}\right) =\left( \frac{%
h^{\prime }-a}{g^{\prime }-d_{1}}\right) ^{\prime }.  \tag{6.14}
\end{equation}%
After twice integration of (6.14), we obtain%
\begin{equation*}
h=\frac{H_{0}}{1+d_{1}^{2}}\left( g-d_{1}y\right) ^{2}+ay+d_{3}\left(
g-d_{1}y\right) +d_{4},
\end{equation*}%
which gives the hypothesis of the theorem. 
\item $f^{\prime \prime }\neq 0.$ By producting (6.8) with $\frac{\left( g^{\prime
}-f^{\prime }\right) ^{3}}{1+\left( f^{\prime }\right) ^{2}}$ and taking
partial derivative with respect to $x$ and $y$ we have%
\begin{equation}
\left. 
\begin{array}{l}
12H_{0}\left[ \frac{\left( g^{\prime }-f^{\prime }\right) f^{\prime \prime
}g^{\prime \prime }}{1+\left( f^{\prime }\right) ^{2}}+\frac{\left(
g^{\prime }-f^{\prime }\right) ^{2}f^{\prime }f^{\prime \prime }g^{\prime
\prime }}{\left[ 1+\left( f^{\prime }\right) ^{2}\right] ^{2}}\right]
=h^{\prime \prime \prime }f^{\prime \prime }+ \\ 
+\left\{ 2g^{\prime }g^{\prime \prime }\left( h^{\prime }-a\right) +\left[
1+\left( g^{\prime }\right) ^{2}\right] h^{\prime \prime }\right\} \left( 
\frac{f^{\prime \prime }}{1+\left( f^{\prime }\right) ^{2}}\right) ^{\prime
}.%
\end{array}%
\right.  \tag{6.15}
\end{equation}%
By dividing (6.15) with $12H_{0}f^{\prime \prime }g^{\prime \prime }$ and
put 
\begin{equation}
A\left( y\right) =2g^{\prime }\left( h^{\prime }-a\right)
+\frac{1+\left( g^{\prime }\right) ^{2}}{g^{\prime \prime}}  h^{\prime \prime }\text{
and }B\left( x\right) =\frac{\left( f^{\prime \prime }/1+\left( f^{\prime
}\right) ^{2}\right) ^{\prime }}{f^{\prime \prime }},  \tag{6.16}
\end{equation}%
we conclude%
\begin{equation}
\frac{\left( g^{\prime }-f^{\prime }\right) }{1+\left( f^{\prime }\right)
^{2}}+\frac{\left( g^{\prime }-f^{\prime }\right) ^{2}f^{\prime }}{\left[
1+\left( f^{\prime }\right) ^{2}\right] ^{2}}=\frac{1}{12H_{0}}\left( \frac{%
h^{\prime \prime \prime }}{g^{\prime \prime }}+AB\right) .  \tag{6.17}
\end{equation}%
Taking partial derivative of (6.17) with respect to $y$ gives%
\begin{equation}
\frac{1}{1+\left( f^{\prime }\right) ^{2}}+\frac{2\left( g^{\prime
}-f^{\prime }\right) f^{\prime }}{\left[ 1+\left( f^{\prime }\right) ^{2}%
\right] ^{2}}=\frac{1}{12H_{0}}\left[ \frac{\left( h^{\prime \prime \prime
}/g^{\prime \prime }\right) ^{\prime }}{g^{\prime \prime }}+\frac{A^{\prime }%
}{g^{\prime \prime }}B\right] .  \tag{6.18}
\end{equation}%
Again taking partial derivative of (6.18) with respect to $y$ and putting%
\begin{equation}
C\left( y\right) =\left[ \frac{\left( h^{\prime \prime \prime }/g^{\prime
\prime }\right) ^{\prime }}{g^{\prime \prime }}\right] ^{\prime }, 
\tag{6.19}
\end{equation}%
we deduce%
\begin{equation}
\frac{f^{\prime }}{\left[ 1+\left( f^{\prime }\right) ^{2}\right] ^{2}}=%
\frac{1}{24H_{0}}\left[ \frac{C}{g^{\prime \prime }}+\frac{\left( A^{\prime
}/g^{\prime \prime }\right) ^{\prime }}{g^{\prime \prime }}B\right] . 
\tag{6.20}
\end{equation}%
The partial derivative of (6.20) with respect to $x$ and $y$ yields%
\begin{equation}
0=\left[ \frac{\left( A^{\prime }/g^{\prime \prime }\right) ^{\prime }}{%
g^{\prime \prime }}\right] ^{\prime }B^{\prime }.  \tag{6.21}
\end{equation}%
From (6.21), we have two possibilities:

\begin{enumerate}
\item[\textbf{(1)}] $B=const.$ Taking partial derivative of (6.18) with respect to $x$ and
producting with $\frac{\left[ 1+\left( f^{\prime }\right) ^{2}\right] ^{3}}{%
f^{\prime \prime }}$ implies the following polynomial equation on $f^{\prime }$:%
\begin{equation*}
(f^{\prime })^{3}-3g^{\prime }(f^{\prime })^{2}-3f^{\prime }+g^{\prime }=0  
\end{equation*}%
which yields a contradiction.

\item[\textbf{(2)}] $B\neq const.$ Then from (6.21), we derive%
\begin{equation}
A=d_{1}\left( g^{\prime }\right) ^{2}+d_{2}g^{\prime }+d_{3}  \tag{6.22}
\end{equation}%
and, from (6.20), $C=d_{4}g^{\prime \prime }$. Then by (6.19) we deduce%
\begin{equation}
\frac{h^{\prime \prime \prime }}{g^{\prime \prime }}=d_{5}\left( g^{\prime
}\right) ^{2}+d_{6}g^{\prime }+d_{7}.  \tag{6.23}
\end{equation}%
Substituting (6.22) and (6.23) into (6.17) yields the polynomial equation on $g^{\prime}$:%
\begin{equation*}
\left. 
\begin{array}{l}
\left\{ \frac{f^{\prime }}{\left[ 1+\left( f^{\prime }\right) ^{2}\right]
^{2}}-\frac{d_{5}+d_{1}B}{12H_{0}}\right\} \left( g^{\prime }\right)
^{2}+\left\{ \frac{1-\left( f^{\prime }\right) ^{2}}{\left[ 1+\left(
f^{\prime }\right) ^{2}\right] ^{2}}-\frac{d_{6}+d_{2}B}{12H_{0}}\right\}
g^{\prime }+ \\ 
+\left\{ \frac{- f^{\prime }}{\left[ 1+\left( f^{\prime }\right) ^{2}\right] ^{2}}-%
\frac{d_{7}+d_{3}B}{12H_{0}}\right\} =0,%
\end{array}%
\right.
\end{equation*}%
where the coefficients must vanish, namely%
\begin{equation}
\left\{ 
\begin{array}{c}
\frac{f^{\prime }}{\left[ 1+\left( f^{\prime }\right) ^{2}\right] ^{2}}=%
\frac{d_{5}+d_{1}B}{12H_{0}}, \\ 
\frac{1-\left( f^{\prime }\right) ^{2}}{\left[ 1+\left( f^{\prime }\right)
^{2}\right] ^{2}}=\frac{d_{6}+d_{2}B}{12H_{0}},\\ 
\frac{- f^{\prime }}{\left[ 1+\left( f^{\prime }\right) ^{2}\right] ^{2}}=\frac{%
d_{7}+d_{3}B}{12H_{0}}.%
\end{array}%
\right. \tag{6.24}
\end{equation}%
It is clear that $d_{1},d_{2},d_{3}$ cannot vanish. By using the first and second equation in (6.24) we obtain the polynomial
equation on $f^{\prime }$:%
\begin{equation*}
d_{1}-d_{1}(f^{\prime })^{2}-d_{2}f^{\prime }=\frac{d_{1}d_{6}-d_{2}d_{5}}{12H_{0}}\left[ 1+\left(
f^{\prime }\right) ^{2}\right] ^{2},
\end{equation*}%
which yields a contradiction. 
\end{enumerate}
\end{itemize}
\end{proof}

\section{Several Remarks}
\begin{enumerate}
\item[\textbf{(1)}] To classify the surfaces of Type IV with arbitrary constant CIGC and CIMC is somewhat complicated, but still it could be a challenging open problem. 

\item[\textbf{(2)}] Isotropic flat translation surfaces of the form (1.1) are generalized cylinders with non-isotropic rulings, that is, one of the translating curves is a non-isotropic line. Therefore, it is clear that there does not exist an isotropic flat surface which belongs to Type IV. In addition, there does not exist
\begin{itemize}
\item a surface of Type I.3, Type II and Type III with non-zero CIGC;

\item an isotropic minimal surface of Type II and Type III.

\end{itemize}

\item[\textbf{(3)}] Isotropic minimal translation surfaces belong to the \textit{family of isotropic Scherk surfaces}. If translating curves are in orthogonal planes, the members of this family are given by (\cite {27})
\begin{itemize}
\item $r\left( x,y\right) =\left( x,y,c\left[ x^{2}-y^{2}\right]
\right) ,$ 

\item $r\left( x,z\right) =\left( x,\frac{1}{c}\log \left\vert 
\frac{cz}{\cos \left( cx\right) }\right\vert ,z\right)$,

\item $r\left( y,z\right) =\frac{1}{2}\left( \frac{1}{c}\log
\left\vert \frac{\cos \left( cz\right) }{\cos \left( cy\right) }\right\vert
,y-z+\pi ,y+z\right)$, $c\in \mathbb{R}-{0}$.
\end{itemize}
If translating curves are in arbitrary planes, the isotropic Scherk surfaces can be described in the explicit forms:
\begin{itemize}
\item $z\left( x,y\right) =c\left[(a_{11}x+a_{12}y)^{2}-\frac{a_{11}^{2}+a_{12}^{2}}{a_{21}^{2}+a_{22}^{2}}(a_{21}x+a_{22}z)^{2}\right]$ (\cite{1}),

\item $y(x,z)=\frac{1}{c}\log \left\vert \frac{\cos \left (\frac {cx}{a_{11}} \right )}{c(a_{21}x+a_{22}z)}\right\vert,$

\item $y(x,z)=\frac{1}{c}\log \left\vert \frac{\cos \left (\frac {ca_{22}}{\det\left( a_{ij}\right) }[a_{11}x+a_{12}z] \right )}{\cos \left ( \frac {c a_{12}}{\det\left( a_{ij}\right)}[a_{21}x+a_{22}z] \right )}\right\vert,$ $c\in \mathbb{R}-{0}$.
\end{itemize}
\end{enumerate}

\end{document}